\documentclass[12pt]{article}

\usepackage{fullpage}

\usepackage{amsmath,amssymb,amsthm}

\usepackage{verbatim}

\usepackage{cancel}

\usepackage{tensor}

\newcommand{\R}{\mathbb{R}}

\newcommand{\lp}{\left (}
\newcommand{\rp}{\right )}

\def\XXint#1#2#3{{\setbox0=\hbox{$#1{#2#3}{\int}$ }
\vcenter{\hbox{$#2#3$ }}\kern-.6\wd0}}

\author{Brian Allen}

\date{Summer 2017}

\title{Long Time Existence of IMCF on Metrics Conformal to Warped Product Manifolds}

\newtheorem{Thm}{Theorem}[section]
\newtheorem{Cor}[Thm]{Corollary}

\newtheorem{Lem}[Thm]{Lemma}

\newtheorem{Def}[Thm]{Definition}
\newtheorem{Ex}[Thm]{Example}

\begin{document}

\maketitle

\begin{abstract}
In this paper we study Inverse Mean Curvature Flow (IMCF) on manifolds that are conformal to a warped product manifold. To this end, we show how the gradient conformal vector field in warped product manifolds is related to the conformal vector field on the conformal metric and use this to gain control of the flow in order to show long time existence as well as asymptotic properties. Connections are made to recent results of the author \cite{BA2,BA3} on stability of the positive mass theorem (PMT) and the Riemannian Penrose inequality (RPI) where long time existence and asymptotic properties of IMCF are important assumptions.
\end{abstract}

\section{Introduction}\label{sec:intro}

Inverse Mean Curvature Flow (IMCF) is defined through a one parameter family of embeddings $\varphi: \Sigma \times [0,T)\rightarrow M^{n+1}$, where $\Sigma^n$ and $M^{n+1}$ are smooth Riemannian manifolds, satisfying the initial value problem
\begin{align}
\begin{cases}
\frac{\partial \varphi}{\partial t}(p,t) = \frac{\nu(p,t)}{H(p,t)}  &\text{ for } (p,t) \in \Sigma \times [0,T)
\\ F(\Sigma,0) = \Sigma_0 
\end{cases}
\end{align}
where $H$ is the mean curvature of $\Sigma_t := \varphi_t(\Sigma)$ and $\nu$ is a consistently chosen normal vector.

Questions of long time existence and asymptotic properties of IMCF were first addressed by Gerhardt \cite{CG1} and Urbas \cite{U} for star-shaped and mean convex hypersurfaces in Euclidean space. Later, star-shaped and mean convex hypersurfaces were studied in Hyperbolic space by Gerhardt \cite{CG2}, rotationally symmetric spaces by Ding \cite{QD} and Scheuer \cite{S}, and warped products by Mullins \cite{Mu}, Zhou \cite{Z} and Scheuer \cite{S2}, to name a few. 

One geometric way of understanding the ambient spaces, $M$, for which long time existence and asymptotic properties of IMCF have been obtained is to notice that the existence of a conformal vector field on $M$, which can be used to define a support function, is key throughout. Then with added structure on the ambient space (globally killing vector field, curvature conditions, conditions on the warping function) the evolution equation for the support function (see Theorem \ref{wevolution}) becomes extremely useful for gaining control on all other important geometric quantities. One goal of this paper is to understand the weakest additional structure we need to assume on $(M, \hat{g})$ in order to obtain long time existence of IMCF.

With this in mind, we would like to show long time existence for manifolds which are conformal to warped product manifolds. We define the warped product metric $\bar{g}$ and the conformal metric $\hat{g}$ as follows on $M=[r_0,\infty) \times \Pi^n$
\begin{align}
\bar{g} &= dr^2 + \lambda(r)^2 \sigma
\\ \hat{g} &= e^{2f} \bar{g}
\end{align}
where $(\Pi, \sigma)$ is a Riemannian manifold, and $\lambda:[r_0,\infty) \rightarrow (0,\infty)$ and $f:M \rightarrow \R$ are smooth functions. The idea is to find sufficient conditions on $\lambda$ and $f$ so that we can prove long time existence of IMCF for strongly star-shaped initial hypersurfaces (This is analagous to the work of Huisken \cite{H} where he defines a notion of strong convexity for mean curvature flow (MCF) in general Riemannian manifolds). More precisely we define what it means to be strongly star-shaped.
\begin{Def}\label{StrongStar}
We say that a hypersurface $\Sigma^n \subset M^{n+1}$ is strongly star shaped with angle $\theta \in [0,\frac{\pi}{2})$ in $(M, \hat{g})$, with respect to the vector field $\eta = \lambda \partial_r$, if $\hat{g}(\nu,\eta) > \cos(\theta)|\eta|$ where $\nu$ is the outward pointing normal vector to $\Sigma$.
\end{Def}
This definition leads to the statement of the following theorem which is the main result of this paper.
\begin{Thm}\label{LTE} Consider $(M, \hat{g})$ and a solution to IMCF, $\Sigma_t$, defined for $t \in [0,T]$ starting at the hypersurface $\Sigma_0$ which is strongly star-shaped for some angle $\theta_1 \in [0,\frac{\pi}{2})$. Choose $f$ and $\lambda$ so that $\exists \theta_2$, $0 \le \theta_2 < \pi/2$, $\theta_1+\theta_2 \le \pi/2$, such that for all $p \in M$ and all $V \in T_pM$, $g(V,\eta) \ge \cos(\theta_1) |\eta|$,
we have that $\lambda' \ge 0$,
\begin{align}
\mathcal{J}(f,\lambda,V):&=\hat{\nabla}_{\eta} \hat{\nabla}f +e^{-2f}\lambda'' \partial_r  +\lambda'\hat{\nabla}f + (|\hat{\nabla}f|^2+\frac{1}{n}) \eta + \frac{1}{n}\hat{Rc}(V,V) \eta
\\\hat{g}(\mathcal{J}(f,\lambda,V), \eta ) &\ge \cos(\theta_2)|\mathcal{J}(f,\lambda,V)| |\eta|,
\end{align}
\begin{align}
\mathcal{G}(f,\lambda):=\hat{\nabla}_{\eta} \hat{\nabla}f +e^{-2f}\lambda'' \partial_r  +\lambda' \hat{\nabla}f + |\hat{\nabla}f|^2 \eta
\\ \hat{g}(\mathcal{G}(f,\lambda), \eta ) \ge \cos(\theta_2)|\mathcal{G}(f,\lambda)| |\eta|,
\end{align}
\begin{align}
0<\delta_1 \le \frac{e^f \lambda '}{\lambda' + \eta(f)} \le \delta_2 < \infty,
\end{align}
and  
\begin{align}
Rc(V,V) \ge -C,
\end{align}
then $\Sigma_t$ exists for all time and remains strongly star shaped with angle $\theta_1$.
\end{Thm}

Once long time existence is established we move on the show some asymptotic estimates which are similar to what one would expect for asymptotically hyperbolic manifolds. These asymptotic estimates aim to impose the weakest conditions possible in order to gain the control necessary to apply the stability results of the author \cite{BA3} in the asymptotically hyperbolic case. 

\begin{Thm}\label{Asymptotics}
If we assume that the hypotheses of Theorem \ref{LTE} hold as well as the following asymptotic conditions:
\begin{align}
&\lambda' \le C_1 |\eta|_{\bar{g}}^{\beta} \text{ for } \beta > 0, 
\\&\|f\|_{C^2} \le C_2|\eta|_{\bar{g}}^{-\alpha} \text{ for }\alpha > 2+\beta,\label{fAsymptotics}
\\ &\frac{\lambda''}{\lambda} \ge 1- C_3|\eta|_{\bar{g}}^{- \gamma} \text{ for } \gamma > 3,
\\ &(n-1) \rho_1(t) \sigma \le Rc^{\Pi} \le (n-1)\rho_2(t) \sigma
\\ &\frac{|\lambda \lambda''+\rho_i(t) -\lambda'^2|}{\lambda^2} \le C_4|\eta|_{\bar{g}}^{-\alpha} \text{ } i=1,2
\end{align}
then
\begin{align}
 C(\alpha, \beta, \gamma, C_1, C_2, C_3) &\le H \le \sqrt{n^2 +C e^{-\frac{\gamma}{n} t} +C_0 e^{-2t}},
 \\\sqrt{1+ |\nabla^B F|^2} &\le   C(\alpha, \beta, \gamma, C_1, C_2, C_3).
\end{align}
\end{Thm}

As an application of this theorem, we would like to consider combining Theorem \ref{LTE} and Theorem \ref{Asymptotics} with the recent work of the author \cite{BA2,BA3} where it is shown that existence results for IMCF on asymptotically flat \cite{BA2} or asymptotically hyperbolic \cite{BA3} manifolds are important for showing $L^2$ stability of the Positive Mass Theorem (PMT) and Riemmanian Penrose Inequality (RPI). This implies that if we can pick $\lambda$ and $f$ which satisfy the hypotheses of Theorem \ref{LTE} and also define an asymptotically flat 3-manifold with positive scalar curvature or an asymptotically hyperbolic 3-manifold with scalar curvatrue greater than $-6$ then we can combine Theorem \ref{LTE} with the results of \cite{BA2, BA3} in order to show $L^2$ stability of the PMT and RPI for these sequences of 3-manifolds. 

It should be mentioned that the stability of the PMT in the case where a metric is conformal to euclidean space has been studied by Dan Lee \cite{L} where uniform convergence of the conformal factor was shown outside a compact set. Similarly, Dahl, Gicquad and Sakovich \cite{DGS} adapted the results of Lee to the asymptotically hyperbolic setting also showing uniform convergence of the conformal factor outside a compact set. The present work implies an extension of these results, in the case where the conditions of Theorem \ref{LTE} and Theorem \ref{Asymptotics} are satisfied, since the present work gives information of what happens inside the compact set at the cost of a weaker form of convergence. Now we give a brief description of the rest of the paper.

In section \ref{sec:CME}, we discuss the geometry of conformal metrics with specific attention on the relationship between the conformal vector field $\eta = \lambda \partial_r$ with respect to $\bar{g}$ and $\hat{g}$. The important calculations of the laplacian of the support function with respect to the metric $\hat{g}$ as well as the evolution equation for the support function under IMCF is computed in this section.

In section \ref{sec:LTE}, we use the assumptions of Theorem \ref{LTE} in order to show long time existence. Here the importance of the assumptions made are clarified and new IMCF estimates for strongly star-shaped hypersurfaces are obtained. One can track which estimates require which assumptions and the section ends with a continuation criterion which ultimately leads to long time existence.

In section \ref{sec:Asymptotics}, we use the assumptions of Theorem \ref{Asymptotics} to obtain asymptotic estimates under the weakest hypotheses possible in order to apply the stability results of the author \cite{BA3}. It should be noted that it would still be interesting to investigate other conditions on $\lambda$ and $f$ would imply stronger asymptotic estimates even for general curvautre functions as in the work of Scheuer \cite{S2}.

In section \ref{sec:Examples}, we explore many classes of examples which satisfy the conditions of Theorem \ref{LTE}. Since these conditions may seem strong at first we found it important to show that these conditions are satisfied for a rich class of examples. We pay special attention to examples relevant for the study of the stability of the Positive Mass Theorem and Riemannian Penrose Inequality. This means that we focus on further requiring that $n=2$ and the scalar curvature $\ge 0$ and $M$ is asymptotically flat, or the scalar curvature $\ge -6$ and $M$ is asymptotically hyperbolic. One will notice that these conditions are easily compatible in the asymptotically hyperbolic case but are fairly incompatible in the asymptotically flat case and so it is best to think of compact regions of manifolds with $R \ge 0$ when combining these results with \cite{BA2}.

\section{Conformal Vector Field Equations}\label{sec:CME}

Remember that a vector field is conformal if 
\begin{align}
\mathcal{L}_{\eta} g = 2 \psi g\label{confvecdef}
\end{align}
for some function $\psi: M \rightarrow \R$ which is called the potential function. Now if we define $\omega$ to be the dual 1-form to the conformal vector field $\eta$, i.e. $\omega(X) = g(X,\eta)$, then we can define the $(1,1)$ tensor $T$ by the formula
\begin{align}
d \omega (X,Y) = 2 g(TX,Y)\label{domegadef}
\end{align}
where $X,Y$ are vector fields on $M$ and we notice that $T$ is skew-symmetric by definition. Using the formula for the exterior derivative of a one form you can find
\begin{align}
d\omega(X,Y) &= \omega(Y)-\omega(X) - \omega([X,Y]) = X(g(Y,\eta)) - Y(g(X,\eta)) - g([X,Y],\eta)\label{ExtderForm}
\end{align}
and so by Koszul's formula we find
\begin{align}
2g(\nabla_X \eta,Y) = (\mathcal{L}_{\eta}g)(X,Y) + d\omega(X,Y)\label{KoszulCons}
\end{align}
Then by combining equations \eqref{confvecdef}, \eqref{domegadef}, \eqref{ExtderForm} and \eqref{KoszulCons} we can deduce the following useful relation
\begin{align}
\nabla_X \eta = \psi X + T(X)\label{confEq}
\end{align}
where $X$ is a vector field on $M$. Equation \eqref{confEq} is very important to gaining control on IMCF as we will see in Lemmas \ref{wLaplacian} and \ref{wevolution} below. When $\eta$ is a closed vector field, i.e. $d \omega = 0$, then $T \equiv 0$ which is also the case for gradient conformal vector fields ($\exists f:M\rightarrow \R$ so that $\eta = \nabla f$). See \cite{Ca, YT}, as well as many other papers, for more about gradient conformal vector fields and closed conformal vector fields but it is important to notice that we are not making those restrictions here and are rather considering general conformal vector fields.

We also note an important equation for the curvature tensor involving a conformal vector field where this equation can be found by covariantly differentiating equation \eqref{confEq}
\begin{align}
R(X,Y)\eta = X(\psi) Y -Y(\psi) X +(\nabla_X T)Y - (\nabla_Y T) X \label{RmConf}
\end{align}

Now consider a warped product metric $\bar{g} = dr^2+\lambda(r)^2 \sigma$, defined on $M^{n+1}=[r_0,\infty) \times \Pi^n$, where $\lambda : [r_0,\infty) \rightarrow (0,\infty)$ is a smooth function. Furthermore, define a metric $\hat{g}$, conformal to $\bar{g}$, by $\hat{g} = e^{2f}\bar{g}$ where $f: N \rightarrow \R$ is a smooth function. We will use $\hat{\nabla}$ and $\bar{\nabla}$ for the covariant derivatives w.r.t. $\hat{g}$ and $\bar{g}$, respectively. Let $\eta=\lambda(r) \partial _r$, where $\partial_r$ is the radial vector field in $M$. For $\eta$ defined this way we can first check that
\begin{align}
\bar{\nabla}_{X}\partial_r =\frac{\lambda'}{\lambda}( X - \bar{g}( \partial_r, X) \partial_r)
\end{align}
which follows by the formula for the shape operator of $\Pi_r := \{r\} \times \Pi$ with respect to the metric $\bar{g}$ and then we find
\begin{align}
\bar{\nabla}_{X}\eta &= \lambda \bar{\nabla}_{X}\partial_r +  \lambda' \bar{g}( \partial_r, X)\partial_r
\\ &=\lambda (\frac{\lambda'}{\lambda} (X - \bar{g}( \partial_r, X) \partial_r)) +  \lambda' \bar{g}( \partial_r, X)\partial_r = \lambda' X
\end{align}
which shows that $\eta$ is a gradient conformal vector field for $\bar{g}$. So if we let $\psi \equiv \lambda'$ then we find
\begin{align}
\mathcal{L}_{\eta}\bar{g}(X,Y) =\bar{g}(\bar{\nabla}_X\eta,Y) + \bar{g}(X,\bar{\nabla}_Y\eta) = 2 \lambda' g(X,Y)
\end{align}
Now we can check that $\eta$ is also conformal for the metric $\hat{g}$
\begin{align}
\mathcal{L}_{\eta}\hat{g} = e^{2f} \mathcal{L}_{\eta}\bar{g} + 2 \eta(f)e^{2f} \bar{g} = 2(\lambda' +  \eta(f))\hat{g}
\end{align}
so that $\hat{\psi}= \lambda' + \eta(f)$ for the metric $\hat{g}$. We note that $\eta$ will not be a gradient conformal vector field with respect to $\hat{g}$ even though $\bar{g}$ is and so we will have to be careful with the tensor $T$ which shows up in the following formula 
\begin{align}
\hat{\nabla}_X \eta = \hat{\psi} X + T(X)\label{hatconfEq}
\end{align}
In fact, we can use the following formula relating covariant derivatives under conformal changes
\begin{align}
\hat{\nabla}_X Y = \bar{\nabla}_X Y +X(f)Y + Y(f)X - \bar{g}(X,Y) \bar{\nabla}f \label{CovDerConChange}
\end{align}
where if $Y = \eta$ in \eqref{CovDerConChange} we find
\begin{align}
\hat{\nabla}_X \eta &= \bar{\nabla}_X \eta +X(f)\eta + \eta(f)X - \bar{g}(X,\eta) \bar{\nabla}f 
\\&=\hat{\psi}X+ X(f)\eta - \bar{g}(X,\eta) \bar{\nabla}f 
\end{align}
which implies
\begin{align}
T(X) =  X(f)\eta - \hat{g}(X,\eta) \hat{\nabla}f \label{SpecTExp}
\end{align}
Now we define the support function of a hypsersurface $\Sigma \subset M$, $w:\Sigma^n \rightarrow \R$, as $w = \hat{g}( \eta, \nu )$ where $\nu$ is the outward pointing unit normal to $\Sigma$. Ultimately we would like to compute the evolution equation of $w$ under IMCF but we start by computing its hypersurface Laplacian.
\begin{Lem}\label{wLaplacian}
Let $\Sigma \subset M$ be a hypersurface with support function $w:\Sigma^n \rightarrow \R$, $w = \hat{g}( \eta, \nu )$, where $\nu$ is the outward pointing unit normal to $\Sigma$. Then we compute it's Laplacian with respect to $\Sigma$
\begin{align}
\Delta w = - |A|^2 w - \hat{Rc}(\nu,\nu)w + \langle \eta, \nabla H \rangle + \hat{\psi} H -n \nu ( \hat{\psi})\label{Lapw}
\end{align}
\end{Lem}
\begin{proof}

Choose an orthonormal frame $\{e_1,...,e_n\}$ for the hypersurface $\Sigma$, normal at the point $p$, so that
\begin{align}
\Delta w|p = e_i(e_i(w)), \hspace{1cm} \nabla_{e_i}e_i|_p = 0
\end{align}
where we will use the convention that summation over repeated indices is implied. Then we find
\begin{align}
e_i(w) &= \hat{g}(\hat{\nabla}_{e_i}\eta,\nu)+\hat{g}(\eta,\hat{\nabla}_{e_i}\nu)
\\&=\hat{g}(T(e_i),\nu)+A(\eta^T, e_i)
\end{align}
where $A$ is the second fundamental form of $\Sigma$ and $^T$ represents projection onto $T_p\Sigma$. Then we find
\begin{align}
\Delta w = \hat{g}(\hat{\nabla}_{e_i}(T(e_i)),\nu) + \hat{g}(T(e_i),\hat{\nabla}_{e_i}\nu)+ (\hat{\nabla}_{e_i}A)(\eta^T, e_i)+A(\hat{\nabla}_{e_i}\eta^T, e_i)+A(\eta^T, \hat{\nabla}_{e_i}e_i)\label{FirstLapw}
\end{align}
First, notice that
\begin{align}
\hat{\nabla}_{e_i}e_i = -H \nu
\end{align}
and hence the last term in \eqref{FirstLapw} is zero. Now by the Codazzi equation we find
\begin{align}
(\hat{\nabla}_{e_i}A)(\eta^T, e_i) = (\hat{\nabla}_{\eta^T}A)(e_i, e_i) + R(e_i,\eta^T,\nu,e_i) = \hat{g}(\nabla H, \eta^T) + Rc(\eta^T,\nu)
\end{align}
and by \eqref{RmConf} we find
\begin{align}
\hat{Rc}(\eta^T,\nu) = \hat{g}(\hat{R}(e_i,\nu)\eta,e_i) &= e_i(\psi)\hat{g}(\nu,e_i) -\nu(\psi)\hat{g}(e_i,e_i) + \hat{g}((\nabla_{e_i}T)\nu,e_i) -\hat{g}((\nabla_{\nu}T)e_i,e_i)
\\&=-\nu(\psi)+ \hat{g}((\nabla_{e_i}T)\nu,e_i) -\hat{g}((\nabla_{\nu}T)e_i,e_i)
\end{align}

\begin{align}
A(\hat{\nabla}_{e_i}\eta^T, e_i)&=A(\hat{\nabla}_{e_i}\eta - \hat{\nabla}_{e_i}(w\nu), e_i)
\\&= \hat{\psi} A(e_i,e_i) + A(T(e_i)^T,e_i)-e_i(w)A(\nu,e_i) - wA(\hat{\nabla}_{e_i}\nu,e_i)
\\&= \hat{\psi} H -w|A|^2
\end{align}
where $A(T(e_i)^T,e_i)=0$ since
\begin{align}
A(T(e_i)^T,e_i) =\hat{g}(T(e_i),\hat{\nabla}_{e_i}\nu)= d\omega(e_i,e_j) A(e_i,e_j)=0
\end{align}
which follows from the symmetry of $A$ and the skew symmetry of $d\omega(e_i,e_j):= \hat{g}(T(e_i),e_j)$. Now we observe
\begin{align}
e_i(\hat{g}(T(\nu),e_i)) &= -e_i(\hat{g}(T(e_i),\nu))
\\ \hat{g}(\hat{\nabla}_{e_i}(T(\nu)),e_i) + \hat{g}(T(\nu),\hat{\nabla}_{e_i}e_i) & = -\hat{g}(\hat{\nabla}_{e_i}(T(e_i)),\nu)-\hat{g}(T(e_i),\hat{\nabla}_{e_i}\nu)
\\ \hat{g}((\hat{\nabla}_{e_i}T)\nu,e_i)+\hat{g}(T(\hat{\nabla}_{e_i}\nu),e_i) -H \hat{g}(T(\nu),\nu) & = -\hat{g}(\hat{\nabla}_{e_i}(T(e_i)),\nu)
\\ \hat{g}((\hat{\nabla}_{e_i}T)\nu,e_i) &= -\hat{g}(\hat{\nabla}_{e_i}(T(e_i)),\nu)
\end{align}
which eliminates every expression in $T$ except for $-\hat{g}((\nabla_{\nu}T)e_i,e_i)$ which we deal with now
\begin{align}
\nu(\hat{g}(T(e_i),e_i) &= 0
\\ \hat{g}(\hat{\nabla}_{\nu}(T(e_i)),e_i) + \hat{g}(T(e_i),\hat{\nabla}_{\nu}e_i) &=0
\\ \hat{g}((\hat{\nabla}_{\nu}T)e_i,e_i) +\hat{g}(T(\hat{\nabla}_{\nu}e_i),e_i) + \hat{g}(T(e_i),\hat{\nabla}_{\nu}e_i) &=0
\\ \hat{g}((\hat{\nabla}_{\nu}T)e_i,e_i)&=0
\end{align}
Now putting this all together we find the desired result \eqref{FirstLapw}.
\end{proof}

\begin{Lem}\label{wevolution}
The evolution equation for the support function $w$ is given by
\begin{align}
\left (\partial_t - \frac{1}{H^2} \Delta \right ) w & = \frac{|A|^2}{H^2} w + \frac{\hat{Rc}(\nu,\nu)}{H^2}w  +n \frac{ \nu ( \hat{\psi} )}{H^2}
\end{align}
\end{Lem}
\begin{proof}
We can compute the time derivative
\begin{align}
\frac{\partial w}{\partial t} &= \langle \frac{\partial \eta}{\partial t} , \nu \rangle + \langle \eta, \frac{\partial \nu}{\partial t}  \rangle 
\\&= \frac{1}{H}\langle  \bar{\nabla}_{\nu} \eta, \nu \rangle + \frac{1}{H^2} \langle \eta, \nabla H \rangle  
\\&= \frac{\hat{\psi}}{H}+ \frac{d \omega (\nu,\nu)}{2H} + \frac{1}{H^2} \langle \eta, \nabla H \rangle
\\&= \frac{\hat{\psi}}{H}+ \frac{1}{H^2} \langle \eta, \nabla H \rangle
\end{align}
where we have used the fact that differential forms are alternating. By combining with \eqref{Lapw} we find the desired evolution equation.
\end{proof}

Now we list some important equations relating the curvature of $\bar{g}$ and $\hat{g}$. One can find equation \eqref{rotsymRicci} in \cite{S,QD} and the equation \eqref{confRiccitang} can be found in \cite{LP}. 
\begin{align}
\bar{Rc}(X,Y)&= Rc^{\Pi}(X^P,Y^P)-n \frac{\lambda''}{\lambda}\bar{g}(X,Y) + (n-1) \left (\frac{\lambda''}{\lambda} - \frac{\lambda'^2}{\lambda^2} \right ) \bar{g}(X^P,Y^P)\label{rotsymRicci}
\\\hat{Rc}(X,Y) &= \bar{Rc}(X,Y) -(n-1) \bar{\nabla}\bar{\nabla}f(X,Y) + (n-1) X(f)Y(f) 
\\&- \bar{\Delta}f\bar{g}(X,Y) -(n-1) |\bar{\nabla}f|^2 \bar{g}(X,Y)\label{confRiccitang}
\end{align}
where $X^P$ represents projection onto the tangent space of $\Pi^n$.

\section{Long Time Existence}\label{sec:LTE}

Now we are interested in trying to control the sign of the term $\nu(\psi)$, in order to gain control of the evolution equation of the support function $w$, which in our case can be expressed in the following form.
\begin{Lem}\label{DerPsiEst1} We can rewrite $\nu(\hat{\psi})$ as
\begin{align}
\nu(\hat{\psi})&=\hat{g}( \hat{\nabla}_{\eta} \hat{\nabla}f + \lambda' \hat{\nabla}f+e^{-2f}\lambda'' \partial_r   + |\hat{\nabla}f|^2 \eta, \nu )
\end{align}
\end{Lem}
\begin{proof}
\begin{align}
\nu(\psi)&= \nu(\eta(f)) +\nu(\lambda')
\\&= \hat{\nabla}\hat{\nabla} f(\eta, \nu) +(\hat{\nabla}_\nu \eta)f+\nu(\lambda')
\\&= \hat{g}( \hat{\nabla}_{\nu} \hat{\nabla}f, \eta ) +T(\nu)(f)+ \hat{\psi}\nu(f)+\nu(\lambda')
\\&=\hat{g}( \hat{\nabla}_{\eta} \hat{\nabla}f + \hat{\psi}\hat{\nabla}f+e^{-2f}\lambda'' \partial_r, \nu ) + \hat{g}(\hat{\nabla}f,T(\nu))
\\&=\hat{g}( \hat{\nabla}_{\eta} \hat{\nabla}f + \hat{\psi}\hat{\nabla}f+e^{-2f}\lambda'' \partial_r + T(\hat{\nabla}f), \nu )
\end{align}
and we can use \eqref{SpecTExp} to find an expression for $T(\hat{\nabla}f)$ as follows
\begin{align}
\hat{g}(T(\hat{\nabla}f),\nu) &= \hat{g}((\hat{\nabla}f)(f)\eta,\nu) - \hat{g}(\hat{\nabla}f,\eta)\hat{g}(\hat{\nabla}f,\nu) = |\hat{\nabla}f|^2w - \hat{g}(\hat{\nabla}f,\eta)\hat{g}(\hat{\nabla}f,\nu)
\end{align}
which allows us to rewrite
\begin{align}
\nu(\psi)&=\hat{g}( \hat{\nabla}_{\eta} \hat{\nabla}f + \hat{\psi}\hat{\nabla}f+e^{-2f}\lambda'' \partial_r  - \hat{g}(\hat{\nabla}f,\eta)\hat{\nabla}f, \nu )+ |\hat{\nabla}f|^2w
\\&=\hat{g}( \hat{\nabla}_{\eta} \hat{\nabla}f + \lambda' \hat{\nabla}f+e^{-2f}\lambda'' \partial_r  , \nu )+ |\hat{\nabla}f|^2w
\end{align}
which is equivalent to the desired expression.
\end{proof}

Notice that $\mathcal{G}$ and $\mathcal{J}$ depend solely on $\hat{g}$ and $\bar{g}$ and so we can impose conditions on these metrics so that $\mathcal{G}$ and $\mathcal{J}$ have desired properties which can be used to show long time existence. The idea is to control the angle, $\theta_1$, between $\mathcal{G}$ (or $\mathcal{J}$) and $\eta$ as well as control the angle, $\theta_2$, between $\nu$ and $\eta$ so that we will be able to deduce control on the angle between $\nu$ and $\mathcal{G}$ (or $\mathcal{J}$).  

\begin{Lem}\label{strongstar}
Choose $f$ and $\lambda$ so that $\exists \theta_1, \theta_2$, $0 \le \theta_1,\theta_2 < \pi/2$, $\theta_1+\theta_2 < \pi/2$, and for all $(x,t) \in \Sigma\times [0,T)$ and all $V \in T_xM$ where $x \in \Sigma_t$ and $g(V,\eta) \ge \cos(\theta_1) |\eta|$ we have that 
\begin{align}
\mathcal{J}(f,\lambda,V):&=\hat{\nabla}_{\eta} \hat{\nabla}f +e^{-2f}\lambda'' \partial_r  +\lambda'\hat{\nabla}f + (|\hat{\nabla}f|^2+\frac{1}{n}) \eta + \frac{1}{n}\hat{Rc}(V,V) \eta
\\\hat{g}(\mathcal{J}(f,\lambda,V), \eta ) &\ge \cos(\theta_2)|\mathcal{J}(f,\lambda,V)| |\eta|
\end{align}
If we assume that when $t = 0$ we have $g(\nu,\eta) > \cos(\theta_1) |\eta|$ then
\begin{align}
w(x,t) > \cos(\theta_1)\left(\min_{\Sigma_0}|\eta|\right) \text{ for } t \in [0,T)
\end{align}
\end{Lem}
\begin{proof}
From the derivation above we have that
\begin{align}
\left (\partial_t - \frac{1}{H^2} \Delta \right ) w & = \frac{|A|^2}{H^2} w + \frac{\hat{Rc}(\nu,\nu)}{H^2}w  +n \frac{ \nu ( \psi )}{H^2} 
\\&\ge n\frac{\hat{g}(\hat{\nabla}_{\eta} \hat{\nabla}f +e^{-2f}\lambda'' \partial_r  +\lambda'\hat{\nabla}f + (|\hat{\nabla}f|^2+\frac{1}{n}) \eta + \frac{1}{n}\hat{Rc}(\nu,\nu) \eta, \nu )}{H^2}
\\&=n\frac{\hat{g}(\mathcal{J}(f,\lambda,\nu),\nu)}{H^2}
\end{align}
For sake of contradiction assume that $w(\bar{x},\bar{t}) = \cos(\theta_1)|\eta|_{(\bar{x},\bar{t})}$ for the first time over $\Sigma\times [0,T)$ at the point $(\bar{x},\bar{t})$ and so
\begin{align}
\frac{\partial w}{\partial t}|_{(\bar{x},\bar{t})} &\le 0
\hspace{1cm} \Delta w|_{(\bar{x},\bar{t})} \ge 0
\hspace{1cm} \left (\partial_t - \frac{1}{H^2} \Delta \right ) w|_{(\bar{x},\bar{t})}  \le 0
\end{align}
Now by assumption we know that $\hat{g}(\mathcal{J}(f,\lambda,\nu),\eta)|_{(\bar{x},\bar{t})} \ge \cos(\theta_2) |\mathcal{J}(f,\lambda,\nu)||\eta|_{(\bar{x},\bar{t})}$ for $\hat{g}(\nu,\eta)|_{(\bar{x},\bar{t})} = \cos(\theta_1)|\eta||_{(\bar{x},\bar{t})}$ and since we have assumed that $\theta_1 + \theta_2 < \frac{\pi}{2}$ we know $\hat{g}(\mathcal{J}(f,\lambda,\nu),\nu)|_{(\bar{x},\bar{t})} > 0$ which implies
\begin{align}
\left (\partial_t - \frac{1}{H^2} \Delta \right ) w|_{(\bar{x},\bar{t})}  > 0
\end{align}
which is a contradiction and hence the lemma holds.
\end{proof}
\textbf{Note:} Lemma \eqref{strongstar} is important because it says that if your inital hypersurface $\Sigma_0$ is strongly star-shaped, i.e. $g(\nu,\eta) > \cos(\theta_1)|\eta|$ for some $\theta_1 < \pi/2$, then it remains strongly star-shaped. This will implicitly be used below.

\begin{Lem}\label{Upperu}
If we define $u = \frac{1}{Hw}$ and choose $f$ and $\lambda$ so that $\exists \theta_1, \theta_2$,$0 \le \theta_1,\theta_2 \le \pi/2$, $\theta_1+\theta_2 \le \pi/2$, and for all $(x,t) \in \Sigma\times [0,T)$ and all $V \in T_xM$ where $x \in \Sigma_t$ and $g(V,\eta) \ge \cos(\theta_1) |\eta|$ we have that 
\begin{align}
\mathcal{J}(f,\lambda,V):&=\hat{\nabla}_{\eta} \hat{\nabla}f +e^{-2f}\lambda'' \partial_r  +\lambda'\hat{\nabla}f + (|\hat{\nabla}f|^2+\frac{1}{n}) \eta + \frac{1}{n}\hat{Rc}(V,V) \eta
\\\hat{g}(\mathcal{J}(f,\lambda,V), \eta ) &\ge \cos(\theta_2)|\mathcal{J}(f,\lambda,V)| |\eta|
\end{align}
as well as
 \begin{align}
\mathcal{G}(f,\lambda):=\hat{\nabla}_{\eta} \hat{\nabla}f +e^{-2f}\lambda'' \partial_r  +\lambda'\hat{\nabla}f + |\hat{\nabla}f|^2 \eta
\\ \hat{g}(\mathcal{G}(f,\lambda), \eta ) \ge \cos(\theta_2)|\mathcal{G}(f,\lambda)| |\eta|
\end{align}
 and $\hat{g}(\nu,\eta) > \cos(\theta_1) |\eta|$ on $\Sigma_0$ then
\begin{align}
u(x,t) \le \max_{\Sigma_0}u
\end{align}
\end{Lem}
\begin{proof}
By combining the standard evolution equation for $H$ with the evolution of $w$, given above, we find the evolution equation for $u$.
\begin{align}
\left (\partial_t - \frac{1}{H^2} \Delta \right )H &= - 2 \frac{|\nabla H|^2}{H^3} - \frac{|A|^2}{H} - \frac{\hat{Rc}(\nu,\nu)}{H}
\\ \left (\partial_t - \frac{1}{H^2} \Delta \right )u & = -\frac{ 2  w}{H} |\nabla u |^2 - \frac{2 }{H^3} \langle \nabla H, \nabla u \rangle -n \frac{ \nu ( \hat{\psi} )}{H^3w^2}
\end{align}
So now we relate the assumptions made in the statement of this lemma to the term $\nu(\hat{\psi})$. Let $\theta_3$ be the largest angle between $\nu$ and $\mathcal{G}(f,\lambda)$ over $\Sigma\times [0,T)$, i.e.
\begin{align}
\nu (\hat{\psi}) = \hat{g}(\mathcal{G}(f,\lambda), \nu ) \ge \cos(\theta_3)|\mathcal{G}(f,\lambda)|
\end{align}
 Now we use that by the assumptions we have $ \hat{g}(\mathcal{G}(f,\lambda), \eta ) \ge \cos(\theta_2)|\mathcal{G}(f,\lambda)| |\eta|$ and $\hat{g}(\nu,\eta) > \cos(\theta_1) |\eta|$ it follows that  $\theta_3 < \theta_1 + \theta_2 \le \frac{\pi}{2}$ and hence $\nu(\psi) \ge 0$ which implies
\begin{align}
\left (\partial_t - \frac{1}{H^2} \Delta \right )u \le -\frac{ 2  w}{H} |\nabla u |^2 - \frac{2 }{H^3} \langle \nabla H, \nabla u \rangle
\end{align}
from which the lemma follows by the maximum principle.
\end{proof}

Now in order to extract, from Lemma \ref{Upperu}, useful information about the lower bound of $H$ we need an upper bound on the norm of $\eta$ along $\Sigma_t$ which is what we proceed to obtain next.

\begin{Lem}\label{Uppereta}
Let $\Sigma_0$ be star-shaped and $\Sigma_t$ the corresponding solution to IMCF in $(M,\hat{g})$. If we assume that $\lambda ' > 0$ and
\begin{align}
0<\delta_1 \le \frac{e^f \lambda '}{\lambda' + \eta(f)} \le \delta_2 < \infty
\end{align}
on $M$, then we find
\begin{align}
\left (\min_{\Sigma_0} |\eta|_{\bar{g}} \right ) e^{\delta_1 t/n} \le |\eta|_{\bar{g}} \le \left (\max_{\Sigma_0} |\eta|_{\bar{g}} \right ) e^{\delta_2 t/n} 
\end{align}
and hence $w \le e^f \displaystyle\left(\max_{\Sigma_0} |\eta|_{\bar{g}} \right ) e^{\delta_2 t/n}$.
\end{Lem}
\begin{proof}

For this we choose a point $x \in \Sigma_t$ so that $|\eta|_{\bar{g}}(x)$ is the maximum of $|\eta|_{\bar{g}}$ over $\Sigma_t$. Then we know that $\Delta |\eta|_{\bar{g}}^2 \le 0$ and we can find the following

\begin{align}
\frac{\partial |\eta|_{\bar{g}}^2}{\partial t} = 2 \bar{g}\left( \frac{\partial \eta}{\partial t}, \eta \right ) &= \frac{2}{H} \bar{g}\left( \bar{\nabla}_{\bar{\nu}}\eta, \eta \right ) = \frac{2}{H}  \bar{g}\left( \lambda' \bar{\nu}, \eta \right )  = \frac{2\lambda' e^{-f}}{H}  w 
\end{align}
where we note that $\bar{\nu} = e^f \nu$.

Now we calculate

\begin{align}
\frac{\eta}{|\eta|_{\bar{g}}}\left (|\eta|_{\bar{g}} \right ) &= |\eta|_{\bar{g}}^{-1} \bar{g}(\bar{\nabla}_{\frac{\eta}{|\eta|_{\bar{g}}}} \eta,\eta) = \lambda' \bar{g}\left( \frac{\eta}{|\eta|_{\bar{g}}}, \frac{\eta}{|\eta|_{\bar{g}}}\right)  = \lambda'
\end{align}
and since by assumption $\lambda' >0$ we have that $\frac{\eta}{|\eta|_{\bar{g}}}\left (|\eta|_{\bar{g}} \right ) > 0$.

Now if we let $\bar{x} \in \Sigma_t$ be the point of maximum for $|\eta|_{\bar{g}}^2$ at time $t$, and hence $|\eta|_{\bar{g}}$, then we have that $e_i (|\eta|_{\bar{g}}) = 0$ for $e_i$ a basis of $T_{\bar{x}}\Sigma_t$ and so we must have that $\frac{\eta}{|\eta|_{\bar{g}}}=\nu$ at $(\bar{x},t)$ in order to reconcile all of the directional derivatives of $|\eta|_{\bar{g}}$.

Let $\Pi= \Pi_{|\eta|_{\bar{g}}(\bar{x})}$ be the level set of $|\eta|_{\bar{g}}^2$ so that $\bar{x} \in \Pi$. Then by the fact that $\bar{x}$ is where the maximum of $|\eta|^2$ occurs we have that $T_{\bar{x}}\Pi = T_{\bar{x}}\Sigma_t$ and $\Sigma_t \subset \Pi$, i.e. contained on the inside, and so by writing $\Sigma_t$ locally as a graph over $\Pi$ we claim we can find that $H(\bar{x}) \ge \tilde{H}(\bar{x})$ where $\tilde{H}$ is the mean curvature of $\Pi$, which we now proceed to show. 

Let $F: \Pi \rightarrow \R$ so that $\Sigma = (F(\Pi),\Pi)$ and define the gradient of $F$ with respect to $\Pi_{|\eta|_{\bar{g}}}$ as $\tilde{\nabla}F = \sigma^{ij} F_j \partial_i$, where $\partial_i$ is a coordinate basis for $\Pi$ and $F_j$ is the coordinate derivative of $F$. Then we can define $\nu = \frac{1}{\sqrt{1+|\tilde{\nabla}F|^2}}\lp-\tilde{\nabla}F + \frac{\eta}{|\eta|_{\hat{g}}}\rp$ (See \cite{CG2} section 1.5 for more on graphs in Riemannian manifolds). Now using the convention where summation over repeated indices is implied and letting $\tilde{H}$ be the mean curvature of $\Pi$ with respect to $\hat{g}$ we can compute 
\begin{align}
H &= \hat{div}(\nu)= \hat{div}(\frac{-\tilde{\nabla}F}{\sqrt{1+|\tilde{\nabla}F|^2}}) + \frac{1}{\sqrt{1+|\tilde{\nabla}F|^2}}\hat{div}(\frac{\eta}{|\eta|_{\hat{g}}})
\\&= \frac{-\hat{g}(\hat{\nabla}_{e_i}\tilde{\nabla}F,e_i)}{\sqrt{1+|\tilde{\nabla}F|^2}} +\frac{\hat{g}(\hat{\nabla}_{e_i}(|\tilde{\nabla}F|^2),\tilde{\nabla}_{e_i}F)}{(1+|\tilde{\nabla}F|^2)^{3/2}}+ \frac{\tilde{H}}{\sqrt{1+|\tilde{\nabla}F|^2}}
\\&= \frac{-\hat{g}(\tilde{\nabla}_{e_i}\tilde{\nabla}F,e_i)}{\sqrt{1+|\tilde{\nabla}F|^2}}-\frac{A_{\Pi_{|\eta|}}(e_i,\tilde{\nabla}F)\hat{g}(\nu,e_i)}{\sqrt{1+|\tilde{\nabla}F|^2}}+\frac{\hat{g}(\tilde{\nabla}_{e_i}(|\tilde{\nabla}F|^2),\tilde{\nabla}_{e_i}F)}{(1+|\tilde{\nabla}F|^2)^{3/2}} + \frac{\tilde{H}}{\sqrt{1+|\tilde{\nabla}F|^2}}
\\&= \frac{-\tilde{Hess}F(e_i,e_i)}{\sqrt{1+|\tilde{\nabla}F|^2}}+\frac{\tilde{Hess}F(\frac{\tilde{\nabla}F}{\sqrt{1+|\tilde{\nabla}F|^2}},\frac{\tilde{\nabla}F}{\sqrt{1+|\tilde{\nabla}F|^2}})}{\sqrt{1+|\tilde{\nabla}F|^2}} +\frac{\tilde{H}}{\sqrt{1+|\tilde{\nabla}F|^2}}
\end{align}
where $A_{\Pi_{|\eta|}}$ is the second fundamental form of $\Pi_{|\eta|}$ and $\tilde{Hess}F$ is the Hessian of $F$ with respect to $\Pi_{|\eta|}$. So if we assume that $\{e_1,...,e_n\}$ diagonalizes the hessian of $F$ at the point $\bar{x}$ then we compute
 \begin{align}
 H &= \frac{1}{\sqrt{1+|\tilde{\nabla}F|^2}}\lp-\tilde{Hess}F(e_i,e_i)+\tilde{Hess}F(\frac{\tilde{\nabla}F}{\sqrt{1+|\tilde{\nabla}F|^2}},\frac{\tilde{\nabla}F}{\sqrt{1+|\tilde{\nabla}F|^2}}) +\tilde{H}\rp
 \\&= \frac{1}{\sqrt{1+|\tilde{\nabla}F|^2}}\lp-\tilde{Hess}F(e_i,e_i)+\tilde{Hess}F(e_i,e_i)\hat{g}\lp\frac{\tilde{\nabla}F}{\sqrt{1+|\tilde{\nabla}F|^2}},e_i\rp^2 +\tilde{H}\rp
 \\&= \frac{1}{\sqrt{1+|\tilde{\nabla}F|^2}}\lp\tilde{Hess}F(e_i,e_i)\lp\hat{g}\lp\frac{\tilde{\nabla}F}{\sqrt{1+|\tilde{\nabla}F|^2}},e_i\rp^2 - 1\rp +\tilde{H}\rp
 \end{align}
from which we deduce that $H(\bar{x}) \ge \tilde{H}(\bar{x})$ since $\hat{g}\lp\frac{\tilde{\nabla}F}{\sqrt{1+|\tilde{\nabla}F|^2}},e_i\rp^2 - 1 \le 0$ and  $\tilde{Hess}F(\bar{x})(e_i,e_i) \le 0$ at the point $\bar{x}$.

 Now we compute $\tilde{H}(\bar{x})$ using the basis $\{e_0=\frac{\eta}{|\eta|}, e_1,...,e_n\}$ where we are using the convention that summation over repeated indices is implied

\begin{align}
\tilde{H}(\bar{x})= \hat{\text{div}}\left ( \frac{\eta}{|\eta|_{\hat{g}}} \right) &= \hat{g}\left(\hat{\nabla}_{e_i}\left (\frac{\eta}{|\eta|_{\hat{g}}} \right),e_i\right )
\\&= \frac{1}{|\eta|_{\hat{g}}} \hat{g}(\hat{\psi} e_i,e_i)+\frac{1}{|\eta|_{\hat{g}}}\hat{g}(T(e_i),e_i) -\frac{\hat{g}(\eta,e_i)}{|\eta|_{\hat{g}}^3} \hat{g}(\hat{\nabla}_{e_i}\eta,\eta)
\\&=\frac{(n+1)\hat{\psi}}{|\eta|_{\hat{g}}} - \frac{\hat{\psi} }{|\eta|_{\hat{g}}^3} \hat{g}( e_i,\eta)g( e_i,\eta) - \frac{1}{|\eta|_{\hat{g}}^3} \hat{g}( e_i,\eta)g( T(e_i),\eta)= \frac{n\hat{\psi}}{|\eta|_{\hat{g}}}
\end{align}

Now we can apply these facts to find the following at the point $\bar{x}$

\begin{align}
\left(\frac{2}{H} e^{-f}\lambda' w  \right )|_{\bar{x}}\le \frac{2\lambda'|\eta|_{\bar{g}} }{\frac{n \hat{\psi}}{|\eta|_{\hat{g}}}}|_{\bar{x}} = \frac{2|\eta|_{\bar{g}}^2}{n} \frac{e^f \lambda' }{\lambda' + \eta(f)}|_{\bar{x}} \le  \frac{2 \delta_2 |\eta|_{\bar{g}}^2}{n}|_{\bar{x}}
\end{align}

which leads to the equation for the evolution of $|\eta|^2$ at points of maximums over $\Sigma_t$

\begin{align}
\frac{\partial }{\partial t} \max_{\Sigma_t}|\eta|_{\bar{g}}^2 \le \frac{2 \delta_2}{n} \max_{\Sigma_t}|\eta|_{\bar{g}}^2
\end{align}

and so by Hamilton's maximum principle we find $|\eta|_{\bar{g}} \le \displaystyle \left  (\max_{\Sigma_0} |\eta|_{\bar{g}} \right ) e^{\delta_2 t/n}$ and hence $w \le \displaystyle \left  (\max_{\Sigma_0} |\eta|_{\bar{g}} \right ) e^{\delta_2 t/n}$. The lower bound follows similarly.
\end{proof}

Now we can combine Lemma \ref{Uppereta} with Lemma \ref{Upperu} in order to find a lower bound on $H$. An upper bound on $H$ is also obtained solely with the lower bound on $Rc$.

\begin{Cor}\label{FirstHEstimates}
Choose $f$ and $\lambda$ so that $\exists \theta_1, \theta_2$ where $0 \le \theta_1,\theta_2 \le \pi/2$, $\theta_1+\theta_2 \le \pi/2$ and for all $(x,t) \in \Sigma\times [0,T)$ and every $V \in T_xM$ where $x \in \Sigma_t$ and $g(V,\eta) \ge \cos(\theta_1) |\eta|$ we have that 
\begin{align}
\mathcal{J}(f,\lambda,V):&=\hat{\nabla}_{\eta} \hat{\nabla}f +e^{-2f}\lambda'' \partial_r  +\lambda'\hat{\nabla}f + (|\hat{\nabla}f|^2+\frac{1}{n}) \eta + \frac{1}{n}\hat{Rc}(V,V) \eta
\\\hat{g}(\mathcal{J}(f,\lambda,V), \eta ) &\ge \cos(\theta_2)|\mathcal{J}(f,\lambda,V)| |\eta|
\end{align}
  as well as
 \begin{align}
\mathcal{G}(f,\lambda):=\hat{\nabla}_{\eta} \hat{\nabla}f +e^{-2f}\lambda'' \partial_r  +\lambda' \hat{\nabla}f + |\hat{\nabla}f|^2 \eta
\\ \hat{g}(\mathcal{G}(f,\lambda), \eta ) \ge \cos(\theta_2)|\mathcal{G}(f,\lambda)| |\eta|
\end{align}
, $\hat{g}(\nu,\eta) > \cos(\theta_1) |\eta|$ on $\Sigma_0$, $\lambda ' > 0$ and
\begin{align}
0<\delta_1 \le \frac{e^f\lambda '}{\lambda' + \eta(f)} \le \delta_2 < \infty
\end{align}
on $M$, then
\begin{align}
H(t)\ge e^{-f}\left(\min_{\Sigma_0}H \right )\left(\min_{\Sigma_0}w \right )\left(\max_{\Sigma_0}w \right )^{-1}e^{-\delta_2 t/n}
\end{align}
If we further assume $Rc(V,V) \ge -C$ for all $V \in T_pM$, $p \in \Sigma_t$, $t \in [0,T]$ and $\hat{g}(V,\eta) \ge \cos(\theta_1)|V||\eta|$ then we find
\begin{align}
H(x,t) \le \max \left ( \max_{\Sigma_t}H, Cn\right ).
\end{align}
\end{Cor}

\begin{proof}
The lower bound follows by combining Lemmas \ref{strongstar}, \ref{Upperu} with Lemma \ref{Uppereta}. The upper bound follows from,
\begin{align}
\left (\partial_t - \frac{1}{H^2} \Delta \right )H &= - 2 \frac{|\nabla H|^2}{H^3} - \frac{|A|^2}{H} - \frac{\hat{Rc}(\nu,\nu)}{H} \le -\frac{1}{n}H+\frac{C}{H},
\end{align}
from which we find,
\begin{align}
\frac{d }{dt}\max_{\Sigma_t}H &\le -\frac{1}{n}\max_{\Sigma_t}H+\frac{C}{\max_{\Sigma_t}H}
\\&=(n\max_{\Sigma_t}H)^{-1} \left (Cn-(\max_{\Sigma_t}H)^2\right).
\end{align}
So by using Hamilton's maximum principle on this autonomous ODE for $\max_{\Sigma_t}H$ we find the upper bound on $H$.
\end{proof}

Now we finish up the proof of long time existence by showing that if  we have uniform upper and lower bounds on mean curvature then the second fundamental form must be bounded.

\begin{Thm} \label{ABounds}
Let $\Sigma_t$ be a smooth solution of IMCF for $t \in [0,T)$ satsifying the bounds $0 < H_0 \le H(x,t) \le H_1$. Then consider the tensor $M_{ij} = HA_{ij}$ where $\{\kappa_1,...,\kappa_n\}$ are the eigenvalues of $M_{ij}$ and $\{\lambda_1,...,\lambda_n\}$ are the eigenvalues of $A_{ij}$. Then we have the following estimates for these eigenvalues

\begin{align}
\kappa_i \le C \hspace{1cm} \lambda_i \le \frac{C}{H_1}
\end{align}
for all $t \in [0,T)$ where the constant $C$ depends on $H_0$, $H_1$, $T$ and the geometry of $\hat{g}$.

\end{Thm}

\begin{proof}
We consider the evolution equation for $M_i^j$ as follows
\begin{align}
\\&\left (\partial _t - \frac{1}{H^2} \Delta \right )  M_i^j  =-2 \frac{\nabla_i H \nabla ^j H}{H^2}    - 2 \frac{1}{H^3} \nabla^k  M_i^j\nabla_k H - 2\tensor{\hat{R}}{_i_0^j_0}  -2\frac{M^{jk}M_{ik}}{H^2}
\\&+ \frac{1}{H} g^{kl} (\hat{\nabla}_i(\tensor{\hat{R}}{_0_l_k^j})-  \hat{\nabla}_k(\tensor{\hat{R}}{_0^j_i_l}) ) + \frac{1}{H^2} g^{kl}g^{pq}(2\tensor{\hat{R}}{_i_k^j_p} M_{ql}- g^{pq}\hat{R}_{kilp} M_q^j- g^{pq}\hat{R}_{kjlp} M_q^i)
\end{align}
Then we use the fact that we know that $M_t \subset B_R(p)$ for some $R > 0$, by the fact that the speed of the flow is bounded, so that $|\hat{Rm}| \le C_R$ and $|\hat{\nabla} \hat{Rm}| \le C_R'$ for some constants $C_R, C'_R > 0$  and  so we find that
\begin{align}
\left (\partial _t - \frac{1}{H^2} \Delta \right ) M_i^j  &\le-2 \frac{\nabla_i H \nabla ^j H}{H^2}    - 2 \frac{1}{H^3} \nabla^k  M_i^j\nabla_k H
\\&  -\alpha M^{jk}M_{ik} + \beta M_i^j + \theta\delta_i^j  
\end{align}
for some constants which depend on the following $\alpha (H_1)$, $\beta (H_0, C_R)$ and $\theta (H_0,C_R,C_R')$.

So now we can compare to the following ODE
\begin{align}
\frac{d\varphi}{dt} = -\alpha \varphi^2 + \beta \varphi + \theta
\end{align}
which either has an upper bound of $\varphi(t) \le \frac{\beta^2 + \sqrt{\beta + 4 \alpha \theta}}{2 \alpha}$ or it is bounded by its value at time 0, i.e. $\varphi (t) \le \varphi (0)$.

So now the result follows by the comparison principle.
\end{proof}

Using Theorem \ref{ABounds} we can also prove a continuation criterion for IMCF.
\begin{Cor}
Let $\Sigma_t\subset M$ be a smooth solution of IMCF. If $H_0 \le H \le H_1$ for $t\in (0,T)$ then we can extend the solution beyond $T$. Furthermore, if $T < \infty$ is the maximal time of existence for the flow $\Sigma_t$ then $1/H \rightarrow \infty$ as $t \rightarrow T$.
\end{Cor}
\begin{proof}
Assume that $T$ is the maximal existence time and using the upper and lower bounds on $H$ combined with Theorem \ref{ABounds} we find $C^2$ control on the solution $\Sigma_t$. Then if we combine with the results of Krylov \cite{K2} we can obtain $C^{2,\alpha}$ control on $\Sigma_t$ and hence if we consider a sequence of times $T_k \in [0,T)$ so that $T_k \nearrow T$ then we know that $\Sigma_{T_k} \rightarrow \Sigma_T$ in $C^{2,\alpha}$ where $\Sigma_T$ is a $C^{2,\alpha}$ hypersurface. Then by short time existence applied to $\Sigma_T$ we can extend the flow beyond time $T$, contradicting the assumption that $T$ was the maximal existence time.
\end{proof}

\section{Asymptotic Analysis}\label{sec:Asymptotics}

In this section we would like to obtain asymptotic estimates that are similar to the asymptotic estimates that we expect for asymptotically hyperbolic manifolds (See \cite{QD,CG3,S}) under the weakest conditions possible. The goal is to obtain the asymptotic estimates required to apply the stability results of the author \cite{BA3} for asymptotically hyperbolic manifolds. 

We start with getting asymptotic control on $\nu(\hat{\psi})$ which shows up in important evolution equations.

\begin{Lem}\label{PsiAsymptotics} If we assume that $\|f\|_{C^2} \le C |\eta|_{\bar{g}}^{-\alpha}$ for $\alpha > 0$ then we can rewrite $\nu(\hat{\psi})$ as
\begin{align}
-C(1+\lambda')|\eta|_{\bar{g}}^{-\alpha} + e^{-2f}\frac{\lambda''}{\lambda} w\le &\nu(\hat{\psi}) \le C(1+\lambda')|\eta|_{\bar{g}}^{-\alpha} + e^{-2f}\frac{\lambda''}{\lambda} w.
\end{align}
Furthermore we can estimate $\delta_1, \delta_2$ from Lemma \ref{Uppereta} since
\begin{align}
\\ 1-C'|\eta|_{\bar{g}}^{-\alpha} \le &\frac{e^f \lambda'}{\lambda' + \eta(f)} \le 1+C'|\eta|_{\bar{g}}^{-\alpha}.
\end{align}
\end{Lem}
\begin{proof}
By going back to Lemma \ref{DerPsiEst1} we know we can express $\nu(\hat{\psi})$ as,
\begin{align}
\nu(\hat{\psi})&= \nu(\eta(f)) + \nu( \lambda') =\nu(\eta(f))+ \hat{g}(e^{-f}\lambda''\partial_r, \nu) = \nu(\eta(f)) +e^{-2f} \frac{\lambda''}{\lambda}w,
\end{align}
from which the first result follows.
For the second we notice
\begin{align}
 1-C'|\eta|_{\bar{g}}^{-\alpha}\le \frac{e^{-C|\eta|_{\bar{g}}^{-\alpha}} \lambda'}{\lambda' + C|\eta|_{\bar{g}}^{-\alpha}}\le \frac{e^f\lambda'}{\lambda' + \eta(f)} \le \frac{e^{C|\eta|_{\bar{g}}^{-\alpha}} \lambda'}{\lambda' - C|\eta|_{\bar{g}}^{-\alpha}} \le 1+C'|\eta|_{\bar{g}}^{-\alpha}
\end{align}
\end{proof}
We now show that asymptotic conditions on $f$ and $\lambda$ imply asymptotic conditions on the Ricci tensor.
\begin{Lem}\label{RcAsymptotics}
If we assume that 
\begin{align}
&\lambda' \le C_1 |\eta|_{\bar{g}}^{\beta} \text{ for } \beta > 0, 
\\&\|f\|_{C^2} \le C_2|\eta|_{\bar{g}}^{-\alpha} \text{ for }\alpha > 2+\beta,\label{fAsymptotics}
\\ &\frac{\lambda''}{\lambda} \ge 1- C_3|\eta|_{\bar{g}}^{- \gamma} \text{ for } \gamma > 3,
\\ &(n-1) \rho_1(t) \sigma \le Rc^{\Pi} \le (n-1)\rho_2(t) \sigma
\\ &\frac{|\lambda \lambda''+\rho_i(t) -\lambda'^2|}{\lambda^2} \le C_4|\eta|_{\bar{g}}^{-\alpha} \text{ } i=1,2
\end{align}
 then we can write
\begin{align}
\hat{Rc} &= -n\hat{g} + B,
\\ |B|&\le C|\eta|^{-\alpha}.
\end{align}

\end{Lem}
\begin{proof}
The equations \eqref{rotsymRicci} and \eqref{confRiccitang},
\begin{align}
\bar{Rc}(X,Y)&= Rc^{\Pi}(X^P,Y^P)-n \frac{\lambda''}{\lambda}\bar{g}(X,Y) + (n-1) \left (\frac{\lambda''}{\lambda} - \frac{\lambda'^2}{\lambda^2} \right ) \bar{g}(X^P,Y^P)
\\\hat{Rc}(X,Y) &= \bar{Rc}(X,Y) -(n-1) \bar{\nabla}\bar{\nabla}f(X,Y) + (n-1) X(f)Y(f) 
\\&- \bar{\Delta}f\bar{g}(X,Y) -(n-1) |\bar{\nabla}f|^2 \bar{g}(X,Y),
\end{align}
combined with \eqref{fAsymptotics} tell us that,
\begin{align}
\hat{Rc} & =  \bar{Rc} + E,
\\|E|&\le C |\eta|^{-\alpha}.
\end{align}
By focusing on $\bar{Rc}$ and choosing $e_i$ a basis of unit vectors with respect to $\hat{g}$ so that $e_0 = \frac{\eta}{|\eta|}$ we find
\begin{align}
\bar{Rc}(e_i,e_i)&\le-ne^{-2f} \frac{\lambda''}{\lambda}\bar{g}(e_i,e_i)+(n-1) \rho_2 e^{-2f} \sigma(e_i^P,e_i^P) + (n-1)e^{-2f} \left (\frac{\lambda''}{\lambda} - \frac{\lambda'^2}{\lambda^2} \right ) \bar{g}(e_i^P,e_i^P)
\\&\le-n\hat{g}(e_i,e_i)+nC_3|\eta|_{\bar{g}}^{-\gamma}\hat{g}(e_i,e_i) +(n-1)e^{-2f} \frac{(\lambda \lambda''+\rho - \lambda'^2)}{\lambda^2}\bar{g}(e_i^P,e_i^P)
\end{align}
where the result now follows.
\end{proof}

\begin{Lem}\label{LowerWAsymptotics2}
If we assume that,
\begin{align}
&\lambda' \le C_1 |\eta|_{\bar{g}}^{\beta} \text{ for } \beta > 0, 
\\&\|f\|_{C^2} \le C_2|\eta|_{\bar{g}}^{-\alpha} \text{ for }\alpha > 2+\beta,\label{fAsymptotics}
\\ &\frac{\lambda''}{\lambda} \ge 1- C_3|\eta|_{\bar{g}}^{- \gamma} \text{ for } \gamma > 3,
\\ &(n-1) \rho_1(t) \sigma \le Rc^{\Pi} \le (n-1)\rho_2(t) \sigma
\\ &\frac{|\lambda \lambda''+\rho_i(t) -\lambda'^2|}{\lambda^2} \le C_4|\eta|_{\bar{g}}^{-\alpha} \text{ } i=1,2
\end{align}
then
\begin{align}
w \ge \lp \min_{\Sigma_0} w\rp e^{t/n} + \frac{n C_9}{\alpha + \gamma -2 -2 \delta_2 - \beta} \lp e^{-\frac{1}{n}\lp \alpha + \gamma -3 -2 \delta_2 - \beta \rp t} - 1\rp
\end{align}
where $\delta_2 \le 1$.
\end{Lem}
\begin{proof}
First notice that be Lemma \ref{PsiAsymptotics} we can choose $t$ large enough and rerun the arguments in Lemma \ref{Uppereta} and Lemma \ref{FirstHEstimates} in order to find that 
\begin{align}
H(x,t) \ge C'' e^{-2\delta_2t/n}
\end{align}
for $\delta_2$ small enough so that $ \alpha > 2\delta_2 + \beta$.
By using the assumptions above we find,
\begin{align}
\left (\partial_t - \frac{1}{H^2} \Delta \right ) w & = \frac{|A|^2}{H^2} w + \frac{\hat{Rc}(\nu,\nu)}{H^2}w  +n \frac{ \nu ( \psi )}{H^2} 
\\&\ge \frac{1}{n} w + ne^{-2f}\frac{\lambda''}{H^2 \lambda} w - \frac{n}{H^2}w -C_4(1+\lambda')|\eta|^{-\alpha}e^{2\delta_2 t/n}
\\&\ge \frac{1}{n}w+ \frac{ne^{-2f}}{H^2}\lp \frac{\lambda''}{\lambda}-1\rp w -C_5(1+\lambda')e^{\frac{1}{n}\lp 2\delta_2  - \alpha\rp t}
\\&\ge \frac{1}{n}w-C_6 e^{\frac{1}{n}\lp 2  - \gamma\rp t}w -C_5(1+\lambda')e^{\frac{1}{n}\lp 2\delta_2  - \alpha\rp t}.
\end{align}
Then by combining with $\lambda' \le C_3 |\eta|^{\beta}$ and Lemma \ref{Uppereta} we find,
\begin{align}
\left (\partial_t - \frac{1}{H^2} \Delta \right ) w & \ge \frac{1}{n}w-C_7 e^{\frac{1}{n}\lp 3  - \gamma\rp t} -C_8e^{\frac{1}{n}\lp 2\delta_2 + \beta  - \alpha\rp t} \ge \frac{1}{n}w -C_9e^{\frac{1}{n}\lp 3 + 2 \delta_2 + \beta  - \alpha - \gamma\rp t}.
\end{align}
Now by applying the comparison principle to this inequality we find
\begin{align}
w &\ge \frac{n C_9}{\alpha + \gamma -2 -2 \delta_2 - \beta} e^{-\frac{1}{n}\lp \alpha + \gamma -3 -2 \delta_2 - \beta \rp t} + \lp \min_{\Sigma_0} w\rp e^{t/n} - \frac{n C_9}{\alpha + \gamma -2 -2 \delta_2 - \beta}
\\& =  \lp \min_{\Sigma_0} w\rp e^{t/n} + \frac{n C_9}{\alpha + \gamma -2 -2 \delta_2 - \beta} \lp e^{-\frac{1}{n}\lp \alpha + \gamma -3 -2 \delta_2 - \beta \rp t} - 1\rp
\end{align}

\end{proof}

Now we notice that if $f$ has compact support, i.e. supp$f\subset\subset M$, then the long time existence results imply that $\Sigma_t$ will be well defined for all time and the upper and lower bounds for $|\eta|$ along $\Sigma_t$ imply that $\Sigma_t$ will eventually move past the support of $f$ and hence the asymptotic estimates of Scheuer \cite{S} hold. The remaining estimates are devoted to the case where $f$ does not have compact support.

\begin{Cor}
Let $\Sigma_t$ be a solution to IMCF so that Lemma \ref{LowerWAsymptotics2} applies then we know $\Sigma_t$ written as a graph over $\Pi_1 = \{(r,\theta_1,...,\theta_n) \in M: r = 1\}$ with graph function $F: \Pi_1 \rightarrow \R$ so that $\Sigma_t = (\Pi_1, F(B_1))$ in $(r,\theta_1,...,\theta_n)$ coordinates. Then we find the estimate,
\begin{align}
\sqrt{1+ |\nabla^{\Pi} F|^2} \le  \frac{\lp\max_{\Sigma_0}|\eta| \rp e^{t/n}}{ \lp \min_{\Sigma_0} w\rp e^{t/n} + \frac{n C_9}{\alpha + \gamma -2 -2 \delta_2 - \beta} \lp e^{-\frac{1}{n}\lp \alpha + \gamma -3 -2 \delta_2 - \beta \rp t} - 1\rp} \le C,
\end{align}
where $\nabla^{\Pi}$ is the covariant derivative on $B_1$.
\end{Cor}
\begin{proof}
We can write ,
\begin{align}
\nu = \frac{e^{-f}}{\sqrt{1+ |\nabla^{\Pi} F|^2}} \lp (0,- \nabla^{\Pi}F) + \partial_r \rp,
\end{align}
which implies,
\begin{align}
w = \hat{g}(\nu,\eta) = \frac{e^{f} \lambda}{\sqrt{1+ |\nabla^{\Pi} F|^2}}.
\end{align}
So by combining with Lemma \ref{LowerWAsymptotics2} we find,
\begin{align}
 \lp \min_{\Sigma_0} w\rp e^{t/n} + \frac{n C_9}{\alpha + \gamma -2 -2 \delta_2 - \beta} \lp e^{-\frac{1}{n}\lp \alpha + \gamma -3 -2 \delta_2 - \beta \rp t} - 1\rp \le  \frac{|\eta|}{\sqrt{1+ |\nabla^{\Pi} F|^2}},
\end{align}
which implies,
\begin{align}
\sqrt{1+ |\nabla^{\Pi} F|^2} \le  \frac{\lp\max_{\Sigma_0}|\eta| \rp e^{t/n}}{ \lp \min_{\Sigma_0} w\rp e^{t/n} + \frac{n C_9}{\alpha + \gamma -2 -2 \delta_2 - \beta} \lp e^{-\frac{1}{n}\lp \alpha + \gamma -3 -2 \delta_2 - \beta \rp t} - 1\rp} \le C.
\end{align}
\end{proof}

\begin{Lem}If we assume that,
\begin{align}
&\lambda' \le C_1 |\eta|_{\bar{g}}^{\beta} \text{ for } \beta > 0, 
\\&\|f\|_{C^2} \le C_2|\eta|_{\bar{g}}^{-\alpha} \text{ for }\alpha > 2+\beta,\label{fAsymptotics}
\\ &\frac{\lambda''}{\lambda} \ge 1- C_3|\eta|_{\bar{g}}^{- \gamma} \text{ for } \gamma > 3,
\\ &(n-1) \rho_1(t) \sigma \le Rc^{\Pi} \le (n-1)\rho_2(t) \sigma
\\ &\frac{|\lambda \lambda''+\rho_i(t) -\lambda'^2|}{\lambda^2} \le C_4|\eta|_{\bar{g}}^{-\alpha} \text{ } i=1,2
\end{align}
then
\begin{align}
H \le \sqrt{n^2 +C e^{-\frac{\alpha}{n} t} +C_0 e^{-2t}},
\end{align}
where $C_0 = \lp \max_{\Sigma_0}H\rp^2 - n^2$.
\end{Lem}
\begin{proof}
Using the decay assumption on $\hat{Rc}$ we find,
\begin{align}
\left (\partial_t - \frac{1}{H^2} \Delta \right )H &= - 2 \frac{|\nabla H|^2}{H^3} - \frac{|A|^2}{H} - \frac{\hat{Rc}(\nu,\nu)}{H}
\\&\le - 2 \frac{|\nabla H|^2}{H^3} - \frac{1}{n}H + \frac{n}{H} +\frac{C_1|\eta|^{-\alpha}}{H},
\end{align}
which implies the ODE,
\begin{align}
\frac{d}{dt}\max_{\Sigma_t}H &\le \frac{1}{n \max_{\Sigma_t}H}\left ( C_2n|\eta|^{-\alpha} + n^2 - \lp\max_{\Sigma_t}H\rp^2 \right )
\\&\le \frac{1}{n \max_{\Sigma_t}H}\left ( C_3e^{-\frac{\alpha}{n}t} + n^2 - \lp\max_{\Sigma_t}H\rp^2 \right ).
\end{align}
We can rewrite this ODE in the following way,
\begin{align}
\frac{d}{dt}\lp \max_{\Sigma_t}H\rp^2 &\le \frac{2}{n}\left ( C_3e^{-\frac{\alpha}{n}t} + n^2 - \lp\max_{\Sigma_t}H\rp^2 \right ).
\end{align}

Now by the comparison principle applied to this ODE we find,
\begin{align}
\lp \max_{\Sigma_t}H\rp^2 &\le n^2 + \lp\lp \max_{\Sigma_0}H\rp^2 - n^2 \rp e^{-2t/n}+ \frac{2 C_3}{2 - \alpha}e^{-\frac{\alpha}{n}t} ,
\end{align}
which implies the desired result.
\end{proof}

\begin{Lem}\label{lowerHAsymptotics} If we assume that,
\begin{align}
&\lambda' \le C_1 |\eta|_{\bar{g}}^{\beta} \text{ for } \beta > 0, 
\\&\|f\|_{C^2} \le C_2|\eta|_{\bar{g}}^{-\alpha} \text{ for }\alpha > 2+\beta,\label{fAsymptotics}
\\ &\frac{\lambda''}{\lambda} \ge 1- C_3|\eta|_{\bar{g}}^{- \gamma} \text{ for } \gamma > 3,
\\ &(n-1) \rho_1(t) \sigma \le Rc^{\Pi} \le (n-1)\rho_2(t) \sigma
\\ &\frac{|\lambda \lambda''+\rho_i(t) -\lambda'^2|}{\lambda^2} \le C_4|\eta|_{\bar{g}}^{-\alpha} \text{ } i=1,2
\end{align}
then
\begin{align}
H \ge C(\alpha, \beta, \gamma, C_1, C_2, C_3).
\end{align}
\end{Lem}
\begin{proof}
Now by using Lemma \ref{PsiAsymptotics} we find,
\begin{align}
 \left (\partial_t - \frac{1}{H^2} \Delta \right )u & = -\frac{ 2  w}{H} |\nabla u |^2 - \frac{2 }{H^3} \langle \nabla H, \nabla u \rangle -n \frac{ \nu ( \hat{\psi} )}{H^3w^2}
 \\&\le  - \frac{2 }{H^3} \langle \nabla H, \nabla u \rangle -\frac{ne^{-2f}}{\max_{\Sigma_t}H^2}  \frac{ \lambda''}{\lambda } u+ \frac{C_4}{H^3 w^2}(1 + \lambda')|\eta|^{-\alpha}   .
\end{align}
Now by using the assumptions of the Lemma we find,
\begin{align}
 \left (\partial_t - \frac{1}{H^2} \Delta \right )u & \le - \frac{2 }{H^3} \langle \nabla H, \nabla u \rangle -\frac{1}{n}  u +C_5 e^{-\frac{\gamma}{n}t}  u + C_6 e^{\frac{1}{n}\lp 1+\beta - \alpha\rp t},
\end{align}
and when combined with Lemma \ref{Uppereta} and Corollary \ref{FirstHEstimates} we find,
\begin{align}
 \left (\partial_t - \frac{1}{H^2} \Delta \right )u & \le - \frac{2 }{H^3} \langle \nabla H, \nabla u \rangle -\frac{1}{n}  u +C_7 e^{-\frac{\gamma}{n}t} + C_6 e^{\frac{1}{n}\lp 1+\beta - \alpha\rp t}.
\end{align}
Now by applying the comparison principle to this ODE we find,
\begin{align}
u \le \lp \max_{\Sigma_0}u+\frac{n C_7}{\gamma - 1}+\frac{n C_6}{\alpha - \beta - 2} \rp e^{-t/n} - \frac{n C_7}{\gamma - 1} e^{-\frac{\gamma}{n}t} - \frac{n C_6}{\alpha - \beta - 2} e^{\frac{1}{n}\lp 1 + \beta - \alpha\rp t},
\end{align}
and hence by unpacking the definition of $u$ we find,
\begin{align}
H &\ge \frac{1}{\lp \max_{\Sigma_0} w\rp e^{t/n}\lp\lp \max_{\Sigma_0}u+\frac{n C_7}{\gamma - 1}+\frac{n C_6}{\alpha - \beta - 2} \rp e^{-t/n} - \frac{n C_7}{\gamma - 1} e^{-\frac{\gamma}{n}t} - \frac{n C_6}{\alpha - \beta - 2} e^{\frac{1}{n}\lp 1 + \beta - \alpha\rp t} \rp}
\\& =  \frac{1}{\lp \max_{\Sigma_0} w\rp \lp\lp \max_{\Sigma_0}u+\frac{n C_7}{\gamma - 1}+\frac{n C_6}{\alpha - \beta - 2} \rp  - \frac{n C_7}{\gamma - 1} e^{\frac{1}{n}\lp 1-\gamma\rp t} - \frac{n C_6}{\alpha - \beta - 2} e^{\frac{1}{n}\lp 2 + \beta - \alpha\rp t} \rp}
\end{align}
\end{proof}

\section{Examples}\label{sec:Examples}
In this section we give some examples to show that the conditions of section \ref{sec:LTE} are satisfied for a rich set of metrics $\hat{g}$. The philosophy throughout is to understand the conditions on $f$ and $\lambda$ solely in terms of coordinate derivatives where the choice of $f$ and $\lambda$ are simply understood. The first example is similar to the cases already handled in a stronger way by Ding \cite{QD}, Scheuer \cite{S}, Mullins \cite{Mu} and Zhou \cite{Z} but we start with this case since we will build on it in later examples.
\begin{Ex} \label{Ex1}
\begin{align}
Rc^{\Pi} \ge n\rho \sigma, f = 0,\lambda \ge 0 , \lambda' > 0 \text{  and  } \lambda'' \ge 0
\\ \lambda '' \lambda + \rho-\lambda'^2 > 0 
\end{align}
In this case we can see that $\Sigma_t$ will remain strongly star-shaped since 
\begin{align}
\hat{g}(\mathcal{J}(f,\lambda,V),\eta) &=\left (\lambda'' + \frac{|\eta|}{n}\right) |\eta| + \frac{1}{n}\bar{Rc}(V,V)|\eta|^2
\\&=\left (\lambda'' + \frac{\lambda}{n}\right) \lambda -\lambda''\lambda + \frac{(n-1)}{n}\left (\lambda''\lambda + \rho-\lambda'^2 \right)\bar{g}(V^S,V^S)> 0
\end{align}
Then we can also obtain lower bound on $H$ since this only requires $\lambda'' \ge 0$.
\begin{align}
\mathcal{G}(f,\lambda) = \lambda''\lambda > 0
\end{align}
Then by the fact that $\lambda ' \ge 0$ we can obtain an upper bound on $|\eta|$ and it is a fairly mild condition to assume that $Rc(V,V) \ge -C$ in this case in order to get an upper bound on $H$. 

So we obtain long time existence in this case but it is difficult to find choices of $\lambda$ which yield asymptotically flat manifolds with positive scalar curvature in this case. It is fairly easy to find examples which are asymptotically hyperbolic with $R \ge -(n+1)n$ such as $\lambda_{l,p,q}(r)= \sinh(lr) + \frac{1}{r^p} +e^{-qr}$ for $0<p \le \frac{1}{n-1}$ and $0<l,q \le \sqrt{n(n-1)}$.
\end{Ex}

\begin{Ex}\label{Ex2}
\begin{align}
\Pi = S^n, f=0, \lambda \ge 0, \lambda' \ge 0 \text{ and  } \lambda'' \le 0
\end{align}
In this case it is quick to check that for choices of $V$ one can find
\begin{align}
\mathcal{J}(f,\lambda,V) \ge 0
\end{align}
but the issue is that we will not be able to obtain the lower bound on $H$, since this requires $\lambda'' \ge 0$ which is crucial to long time existence. Though, by the assumption that $\lambda ' \ge 0$ we can obtain an upper bound on $|\eta|$ and it is a fairly mild condition to assume that $Rc(V,V) \ge -C$ in this case in order to get an upper bound on $H$.
\end{Ex}

\begin{Ex}\label{Ex3}
$\Pi = S^n$, $\lambda=r$, $r^2f_{rr} + rf_r > \delta > 0$ and 
\begin{align}
\hat{g}(\mathcal{G}(f,r),\eta)&=\hat{g}(\hat{\nabla}_{\eta} \hat{\nabla}f  + \hat{\nabla}f + |\hat{\nabla}f|^2 \eta,\eta)
\\&=e^{2f}\bar{g}(\hat{\nabla}_{\eta} \hat{\nabla}f  + \hat{\nabla}f + |\hat{\nabla}f|^2 \eta,\eta)
\\&= e^{2f}\left (\eta(\eta(f)) -(\hat{\nabla}_{\eta}\eta)f + \eta(f) +|\hat{\nabla}f|^2\bar{g}(\eta,\eta)\right)
\\&= e^{2f}\left (r\partial_r(r\partial_r(f)) -\hat{\psi}\eta(f) + \eta(f) +|\hat{\nabla}f|^2\bar{g}(\eta,\eta)\right)
\\&= e^{2f}\left (r^2f_{rr} + rf_r +|\hat{\nabla}f|^2\bar{g}(\eta,\eta) - \eta(f)^2\right)
\\&= e^{2f}\left (r^2f_{rr} + rf_r +e_i(f)^2\right) \ge e^{2f}\left (r^2f_{rr} + rf_r\right)
\end{align}
where $e_i$ is a orthonormal basis of $S^n$ w.r.t $\hat{g}$. So we see that we just need to choose $f$ so that $rf_{rr} + f_r > \delta$ which is nice since we have freedom in how we choose $f$ in the $S^n$ direction.

The idea of the next calculation is to rewrite $\hat{g}(\mathcal{J}(f,\lambda,V),\eta)$ solely in terms of $f$ and coordinate derivatives of $f$ so that we can get an idea of the freedom we have to choose $f$ while requiring $\hat{g}(\mathcal{J}(f,\lambda,V),\eta)$ to be bounded away from zero.
\begin{align}
\hat{g}&(\mathcal{J}(f,r,V),\eta)=\hat{g}(\hat{\nabla}_{\eta} \hat{\nabla}f +\hat{\nabla}f + (|\hat{\nabla}f|^2+\frac{1}{n}) \eta + \frac{1}{n}\bar{Rc}(V,V) \eta,\eta)
\\&= e^{2f}\left (r^2f_{rr} + rf_r +|\hat{\nabla}f|^2\bar{g}(\eta,\eta) - \eta(f)^2 +\frac{r^2}{n} +  \frac{r^2}{n}\bar{Rc}(V,V)\right)
\\&= e^{2f}\left (r^2f_{rr} + rf_r +r^2|\bar{\nabla}f|^2-\frac{(n-1)}{n}r^2|\bar{\nabla}f|^2\bar{g}(V,V)+\frac{(n-1)}{n} r^2V(f)^2 - \eta(f)^2 +\frac{r^2}{n}\right)
\\& - \frac{1}{n}r^2e^{2f}\left ((n-1)\bar{\nabla}\bar{\nabla}f(V,V) + \bar{\Delta} f\bar{g}(V,V)\right)
\end{align}
So if we assume that $V$ is a vector of unit length w.r.t. $\bar{g}$ then we find
\begin{align}
\hat{g}(\mathcal{J}(f,r,V),\eta)&= e^{2f}\left (r^2f_{rr} + rf_r +\frac{r^2}{n}|\bar{\nabla}f|^2+\frac{(n-1)}{n} r^2V(f)^2 - \eta(f)^2 +\frac{r^2}{n}\right)
\\& - \frac{1}{n}r^2e^{2f}\left ((n-1)\bar{\nabla}\bar{\nabla}f(V,V) + \bar{\Delta} f\right)
\end{align}
and furthermore, we can compute $\bar{\Delta} f$
\begin{align}
\bar{\Delta} f &= f_{rr} + \frac{f_{e_ie_i}}{r^2} -(\bar{\nabla}_{\partial_r}\partial_r)(f)-\frac{1}{r^2} \sigma^{ij}(\bar{\nabla}_{e_i}e_j)(f)=f_{rr} + \frac{\tilde{\Delta}f}{r^2}
\end{align}
where $\sigma, \tilde{\nabla}$ represents the metric and covariant derivatives on the sphere $S^n$, respectively, which depend on the basis, $\{e_1,...e_n\}$ chosen.

So now if we let $n=2$ we find
\begin{align}
\hat{g}&(\mathcal{J}(f,r,V),\eta)=r^2e^{2f}\left (f_{rr} + \frac{f_r}{r} +\frac{1}{2}(|\bar{\nabla}f|^2+V(f)^2 - 2\partial_r(f)^2 +1)\right )
\\&-\frac{r^2}{2}e^{2f}\left ( V(V(f) - (\bar{\nabla}_V V)(f)+f_{rr}+ \frac{\tilde{\Delta}f}{r^2}\right)
\\&=\frac{r^2}{2}e^{2f}\left (f_{rr} +2\frac{f_r}{r} +(|\bar{\nabla}f|^2+V(f)^2 - 2\partial_r(f)^2 +1)-V(V(f)) + (\bar{\nabla}_V V)(f) - \frac{\tilde{\Delta}f}{r^2}\right )
\end{align}

Now if we let $\displaystyle V = V_0\partial_r + V_i \partial_{e_i}$ where $V_0^2 + V_i^2 = 1$ then we find
\begin{align}
V(f)^2 &= V_0^1f_r^2 + V_i^2 f_{e_i}^2 
\\V(V(f)) &= (V_0\partial_r + V_i \partial_{e_i})((V_0\partial_r + V_j \partial_{e_j})(f))
\\&= V_0^2 f_{rr} + V_0V_{0r} f_r + 2V_0V_jf_{re_j}+V_0V_{jr} f_{e_j}+V_iV_{0i}f_r + V_iV_j f_{e_ie_j} + V_i V_{je_i}f_{e_j}
\\&= V_0^2 f_{rr}+ 2V_0V_jf_{re_j} + V_iV_j f_{e_ie_j} + V_0V_{0r} f_r +V_0V_{jr} f_{e_j}+V_iV_{0e_i}f_r + V_i V_{je_i}f_{e_j}
\\(\bar{\nabla}_V V)(f)&=(\bar{\nabla}_{V_0\partial_r + V_i \partial_{e_i}} (V_0\partial_r + V_j \partial_{e_j})(f)
\\&=V_0^2(\bar{\nabla}_{\partial_r} \partial_r)(f)+V_0V_j(\bar{\nabla}_{\partial_r} e_j)(f)+V_iV_0(\bar{\nabla}_{e_i} \partial_r)(f)+V_iV_j(\bar{\nabla}_{e_i} e_j)(f)
\\&+ V_0V_{0r} f_r+V_iV_{0e_i}f_r +V_0V_{jr} f_{e_j} + V_i V_{je_i}f_{e_j}
\\&=2V_0V_i\frac{f_{e_i}}{r}+V_iV_j(\tilde{\nabla}_{e_i} e_j)(f)+ V_0V_{0r} f_r+V_iV_{0e_i}f_r +V_0V_{jr} f_{e_j} + V_i V_{je_i}f_{e_j}
\end{align}
where $\tilde{\nabla}$ represents covariant derivatives on the sphere $S^n$ which depend on the basis chosen. Now we can find
\begin{align}
\hat{g}(\mathcal{J}(f,r,V),\eta)&=\frac{r^2}{2}e^{2f}\left (1+(V_0^2-1) f_r^2 +(\frac{1}{r^2}+ V_i^2)f_{e_i}^2 +(1-V_0^2) f_{rr}- 2V_0V_jf_{re_j} - V_iV_j f_{e_ie_j})\right)
\\&+\frac{r^2}{2}e^{2f}\left (2\frac{f_r}{r}+2V_0V_i\frac{f_{e_i}}{r}+V_iV_j(\tilde{\nabla}_{e_i} e_j)(f)- \frac{\tilde{\Delta}f}{r^2}\right)
\end{align}
In particular we have the formula for the scalar curvature 
\begin{align} 
\hat{R} &=  -ne^{-2f} \left( 2 \bar{\Delta}f  +(n-1) |\bar{\nabla} f|^2 \right)
\\&=-ne^{-2f} \left(2f_{rr} + 2\sigma^{ij}\frac{f_{e_ie_j}}{r^2}-2\frac{1}{r^2} \delta^{ij}(\tilde{\nabla}_{e_i}e_j)(f) + (n-1) f_r^2 + (n-1)\frac{f_{e_i}^2}{r^2}\right )
\end{align}
where we see by the other conditions that it is difficult to have $\hat{R} \ge 0$ but it is certainly possible to impose the condition that $\hat{R} \ge -n(n+1)$.
\end{Ex}

\begin{Ex}\label{Ex4}
$Rc^{\Pi} \ge n\rho \sigma$, $\lambda >0$, $f_{rr} + \frac{\lambda'}{\lambda}f_r +e^{-2f}\frac{\lambda''}{\lambda}\ge \delta > 0$
\begin{align}
\hat{g}&(\mathcal{G}(f,r),\eta)=\hat{g}(\hat{\nabla}_{\eta} \hat{\nabla}f +e^{-2f}\lambda'' \partial_r + \lambda'\hat{\nabla}f + |\hat{\nabla}f|^2 \eta,\eta)
\\&=e^{2f}\bar{g}(\hat{\nabla}_{\eta} \hat{\nabla}f+e^{-2f}\lambda'' \partial_r + \lambda'\hat{\nabla}f + |\hat{\nabla}f|^2 \eta,\eta)
\\&= e^{2f}\left (\eta(\eta(f)) -(\hat{\nabla}_{\eta}\eta)f + \lambda'\eta(f) +|\hat{\nabla}f|^2\bar{g}(\eta,\eta)+e^{-2f}\lambda'' \lambda\right)
\\&= e^{2f}\left (\lambda\partial_r(\lambda\partial_r(f)) -\hat{\psi}\eta(f) + \lambda'\eta(f) +|\hat{\nabla}f|^2\bar{g}(\eta,\eta)+e^{-2f}\lambda'' \lambda \right)
\\&= e^{2f}\left (\lambda^2f_{rr} + \lambda\lambda'f_r +|\hat{\nabla}f|^2\bar{g}(\eta,\eta) - \eta(f)^2+e^{-2f}\lambda'' \lambda\right)
\\&= e^{2f}\left (\lambda^2f_{rr} + \lambda\lambda'f_r +e_i(f)^2+e^{-2f}\lambda'' \lambda\right) \ge e^{2f}\lambda^2\left ( f_{rr} + \frac{\lambda'}{\lambda}f_r +e^{-2f}\frac{\lambda''}{\lambda} \right) \label{GenCond1}
\end{align}
where again we see the freedom in the $\Pi^n$ direction. 

The idea of the next calculation is to rewrite $\hat{g}(\mathcal{J}(f,\lambda,V),\eta)$ solely in terms of $f$ and $\lambda$ and coordinate derivatives of $f$ and $\lambda$ so that we can get an idea of the freedom we have to choose $\lambda$ and $f$ while requiring $\hat{g}(\mathcal{J}(f,\lambda,V),\eta)$ to be bounded away from zero.
\begin{align}
\hat{g}(\mathcal{J}(f,\lambda,V),\eta)&=\hat{g}(\hat{\nabla}_{\eta} \hat{\nabla}f+e^{-2f}\lambda''\partial_r +\lambda'\hat{\nabla}f + (|\hat{\nabla}f|^2+\frac{1}{n}) \eta + \frac{1}{n}Rc(V,V) \eta,\eta)
\\&= e^{2f}\left (\lambda^2f_{rr} + \lambda\lambda'f_r +|\hat{\nabla}f|^2\bar{g}(\eta,\eta)-\frac{(n-1)}{n}\lambda^2|\bar{\nabla}f|^2\bar{g}(V,V)\right)
\\&+ e^{2f}\left (\frac{(n-1)}{n} \lambda^2V(f)^2 - \eta(f)^2+e^{-2f}\lambda'' \lambda+\frac{\lambda^2}{n}\right)
\\& - \lambda^2e^{2f}\left (\frac{(n-1)}{n}\bar{\nabla}\bar{\nabla}f(V,V) + \frac{1}{n}\bar{\Delta} f\bar{g}(V,V)\right) + e^{2f}\frac{\lambda^2 }{n}\bar{Rc}(V,V)
\end{align}
Now we can compute $\bar{\Delta} f$
\begin{align}
\bar{\Delta} f &= f_{rr} + \frac{f_{e_ie_i}}{\lambda^2} -(\bar{\nabla}_{\partial_r}\partial_r)(f)-\frac{1}{\lambda^2} \sigma^{ij}(\bar{\nabla}_{e_i}e_j)(f)=f_{rr} +\frac{\tilde{\Delta}f}{\lambda^2}\label{LapLam}
\end{align}
where $\sigma, \tilde{\nabla}$ represents the metric and covariant derivatives on $\Pi^n$, respectively, which depend on the basis, $\{e_1,...e_n\}$ chosen.

Furthermore, if assume that $V$ is unit length w.r.t. $\bar{g}$ and use \eqref{LapLam} then we find
\begin{align}
&\hat{g}(\mathcal{J}(f,\lambda,V),\eta)\ge e^{2f}\left (\lambda^2f_{rr} + \lambda\lambda'f_r +\frac{\lambda^2}{n}|\bar{\nabla}f|^2+\frac{(n-1)}{n} \lambda^2V(f)^2 - \eta(f)^2+e^{-2f}\frac{\lambda''}{\lambda} +\frac{\lambda^2}{n}\right)
\\& - \lambda^2e^{2f}\left (\frac{(n-1)}{n}V(V(f)) -\frac{(n-1)}{n}(\bar{\nabla}_V V)f + \frac{1}{n}f_{rr} + \frac{\tilde{\Delta}f}{n\lambda^2}\right)
\\ &+e^{2f}\lambda^2 \left(-\frac{\lambda''}{\lambda} + \frac{(n-1)}{n} (\frac{\lambda''}{\lambda} +\frac{\rho-\lambda'^2}{\lambda^2})\bar{g}(V^S,V^S) \right)
\\&= \lambda^2 e^{2f}\left (\frac{\lambda'}{\lambda}f_r +\frac{1}{n}|\bar{\nabla}f|^2-\frac{(n-1)}{n}V(V(f))+\frac{(n-1)}{n}(\bar{\nabla}_V V)f- \frac{\tilde{\Delta}f}{n\lambda^2}- (1-e^{-2f})\frac{\lambda''}{\lambda} \right)
\\&+\lambda^2 e^{2f}\left ( \frac{(n-1)}{n}f_{rr}+\frac{(n-1)}{n} V(f)^2 -f_r^2 +\frac{1}{n}+ \frac{(n-1)}{n} (\frac{\lambda''}{\lambda} +\frac{\rho-\lambda'^2}{\lambda^2})\bar{g}(V^S,V^S)\right)
\end{align}
and if we further assume that $n=2$ and let $\displaystyle V = V_0\partial_r + V_i \partial_{e_i}$ where $V_0^2 + V_i^2 = 1$ then we find
\begin{align}
&\hat{g}(\mathcal{J}(f,\lambda,V),\eta)\ge \frac{\lambda^2}{2} e^{2f}\left (2\frac{\lambda'}{\lambda}f_r+|\bar{\nabla}f|^2-V(V(f))+(\bar{\nabla}_V V)f- \frac{\tilde{\Delta}f}{\lambda^2}- (1-e^{-2f})\frac{\lambda''}{\lambda} \right)
\\&+\frac{\lambda^2}{2} e^{2f}\lp f_{rr}+1+V(f)^2 -2f_r^2 +\left (\frac{\lambda''}{\lambda} +\frac{\rho-\lambda'^2}{\lambda^2}\right)|V^S|_{\bar{g}}^2\rp
\\&= \frac{\lambda^2}{2} e^{2f}\left (2\frac{\lambda'}{\lambda}f_r+(1-V_0^2) f_{rr}- 2V_0V_jf_{re_j} - V_iV_j f_{e_ie_j}+V_iV_j(\tilde{\nabla}_{e_i} e_j)(f)- \frac{\tilde{\Delta}f}{\lambda^2} \right)
\\&+\frac{\lambda^2}{2} e^{2f}\left (1+(V_i^2+e^{-2f} - 1)\frac{\lambda''}{\lambda} +V_i^2\frac{\rho-\lambda'^2}{\lambda^2}+(V_0^2-1)f_r^2+(\frac{1}{\lambda^2} +V_i^2)f_{e_i}^2+2V_0V_i\frac{\lambda' f_{e_i}}{\lambda}\right)
\end{align}
by which we see that the conditions $\lambda'' \ge 0$, $\rho > 0$, $\lambda'^2 \le \rho$, $f_r \ge 0$, $f_{rr}\ge 0$, and the $\Pi^n$ hessian of $f$ being negative are preferred in this case in order to bound $\hat{g}(\mathcal{J}(f,\lambda,V),\eta)$ away from zero. Notice that these preferences agree with the condition \eqref{GenCond1} in order to get long time existence. Any problem terms which involve derivatives in the $\Pi^n$ direction can be chosen as small as needed to preserve the bounds from the radial choices and $V$ can be chosen as close to $\partial_r$ as necessary.

Now we calculate the scalar curvature using the formulas found in \cite{LP}
\begin{align}
\hat{R} &= e^{-2f} \left ( \bar{R} - 2n \bar{\Delta}f - n(n-1) |\bar{\nabla} f|^2 \right)
\\	&\ge e^{-2f} \left ( \frac{n(n-1)(\rho - \lambda'^2)}{\lambda^2} -\frac{2n\lambda''}{\lambda}\right)
\\&-e^{-2f} \left (  2n \left(f_{rr} + \frac{\tilde{\Delta}f}{\lambda^2}\right) + n(n-1) (f_r^2 + \frac{f_{e_i}^2}{\lambda^2}) \right)
\end{align}
where we see now that it is reasonable to choose an $f$ and $\lambda$ so that $\hat{R} \ge 0$ in this case because of the extra freedom that $\lambda$ provides in the equation above. Some of the conditions which cause $\hat{R} \ge 0$ are at odds with the conditions needed above but there is enough freedom to choose combinations of preferences to satisfy all conditions necessary to find existence of solutions to IMCF for some $T>0$ which is relevant for proving stability of the PMT and RPI as in \cite{BA2}.

More excitingly we see that one can easily satisfy the condition $\hat{R} \ge -6$ in order to get long time existence of IMCF for asymptotically hyperbolic manifolds. When you combine this fact with the asymptotic results of section \ref{sec:Asymptotics} and the results of the author \cite{BA3} one can show stability of the PMT and RPI for asymptotically hyperbolic manifolds.
\end{Ex}

\end{document}